\documentclass[12pt,a4paper,leqno]{amsart}
\usepackage{amsmath,amsthm,amssymb}
\usepackage[cp1250]{inputenc}

\usepackage{tikz}
\usetikzlibrary{patterns}
\usetikzlibrary{fadings}

\newcommand{\Ea }{\mathbb E}

\newcommand{\Ia }{\mathbb I}
\newcommand{\Ka }{\mathbb K}
\newcommand{\La }{\mathbb L}

\newcommand{\Pa }{\mathbb P}
\newcommand{\Qa }{\mathbb Q}
\newcommand{\Ra }{\mathbb R}
\newcommand{\Sa }{\mathbb S}
\newcommand{\Ca }{\mathbb C}

\newcommand{\bfp}{\mathbf{p}}
\newcommand{\bfx}{\mathbf{x}}
\newcommand{\bfy}{\mathbf{y}}
\newcommand{\bfz}{\mathbf{z}}

\newcommand{\Baa }{\mathcal B}

\newcommand{\cl}{\operatorname{cl}}
\newcommand{\dist}{\operatorname{dist}}

\newcommand{\rc}{\textrm{\texttt{RC}}}

\renewcommand{\int}{\operatorname{int}}

\newcommand{\ot}[5]{\draw[dashed,fill=gray,opacity=.3]  (#1,#2) -- (#1,#2+#3) -- (#1+#3,#2+#3) -- (#1+#3,#2);
  \draw[thick] (#1,#2+#3) -- (#1,#2) -- (#1+#3,#2);
  \draw[fill] (#1,#2) circle [radius=.07];
  \draw[fill=white] (#1,#2+#3) circle [radius=.035];
  \draw[fill=white] (#1+#3,#2+#3) circle [radius=.035];
  \draw[fill=white] (#1+#3,#2) circle [radius=.035];
  \draw (#1-.2,#2-.4) node[anchor=west,scale=.7,fill=white] {$#4$};
  \draw (2*#1/2+#3/2,2*#2/2+#3/2) node [fill=white,scale=.7] {$#5$};} 

\newcommand{\otc}[4]{\draw[dotted]  (#1,#2) -- (#1,#2+#3) -- (#1+#3,#2+#3) -- (#1+#3,#2);
  \draw[dotted] (#1,#2+#3) -- (#1,#2) -- (#1+#3,#2);
  \draw[fill] (#1,#2) circle [radius=.07];
  \draw[fill=white] (#1+#3,#2+#3) circle [radius=.035];
\draw (#1-.2,#2-.4) node[anchor=west,scale=.7] {$#4$};}

\newcommand{\otd}[4]{\draw[dashed]  (#1,#2) -- (#1,#2+#3) -- (#1+#3,#2+#3) -- (#1+#3,#2);
  \draw[thick] (#1,#2+#3) -- (#1,#2) -- (#1+#3,#2);
  \draw[fill] (#1,#2) circle [radius=.07];
  \draw[fill=white] (#1,#2+#3) circle [radius=.035];
  \draw[fill=white] (#1+#3,#2+#3) circle [radius=.035];
  \draw[fill=white] (#1+#3,#2) circle [radius=.035];
\draw (#1-.2,#2-.4) node[anchor=west,scale=.7] {$#4$};}

\newtheorem{thm}{Theorem}
\newtheorem*{thm*}{Theorem}
\newtheorem{pro}[thm]{Proposition}
\newtheorem{lem}[thm]{Lemma}

\newtheorem{cor}[thm]{Corollary}
\newtheorem{fac}[thm]{Fact}

\title{On \rc-spaces}
\subjclass[2000]{Primary: 54; Secondary: 05}
\keywords{Regular closed sets, Niemytzki plane, Sorgenfrey plane}

\author{Wojciech Bielas}
\address{Institute of Mathematics, University of Silesia, ul. Bankowa 14, 40-007 Katowice}
\email{wojciech.bielas@us.edu.pl}

\author{Szymon Plewik}
\address{Institute of Mathematics, University of Silesia, ul. Bankowa 14, 40-007 Katowice}
\email{plewik@math.us.edu.pl}

\begin{document}

\maketitle

\begin{abstract}
 Following Frink's characterization of completely regular spaces, we say that a regular $T_1$-space is an  \rc-space whenever the family of all regular open sets constitutes a regular normal base. Normal spaces are \rc-spaces and there exist completely regular spaces which are not \rc-spaces. So the question arises, which of the known examples of completely regular and not normal  spaces are \rc-spaces. We show that the Niemytzki plane and the Sorgenfrey plane are \rc-spaces. \end{abstract}

\section{Introduction} \label{s1}

The aim of this note is to examine a topological space in which all regular open sets form a  normal base.
For a normal space, the family of all open sets is a  normal base, therefore we  narrow our attention to a completely regular space which is not normal.
  O. Frink defined a normal base and, using this notion, gave the following characterization of a completely regular space,  compare \cite{fri} or  \cite[Exercise 1.5.G]{en}.

\begin{thm*}
A $T_1$-space $X$ is completely regular if and only if there exists a base $\Baa$ satisfying the following conditions.
\vspace{-0.5cm} \begin{enumerate}
\item If  $x \in U \subset  X$, where $U$ is open, then  there exists a set $V\in\Baa$ such that $x \in X \setminus  V \subset U$. 
\item If $U,V\in \Baa$ and $U\cup V =X$, then there exist sets $U^*, V^* \in \Baa$ such that $X \setminus V \subset U^* \subset X \setminus  V^* \subset U.$  
\end{enumerate}
\end{thm*} \vspace{-0.5cm}
Note that the base consisting of all co-zero sets satisfies Frink's characterization, but repeating a proof of Urysohn's lemma one obtains a proof of Frink's theorem.
Later, A. Zame \cite{zam} formally defined a regular normal base as a ring of regular closed sets which is also a normal base, see also \cite{mis}.
We adapt the concept of a normal base in terms of open sets, omitting the assumption that a base must be a ring of sets.
A base of $T_1$-space is a \emph{regular normal base}, whenever it consists  of regular open sets which satisfies conditions (1) and (2) in Frink's theorem. Thus, any space with a regular normal base is completely regular. 

Our notation is mostly standard and follows \cite{en}. However, names of considered  topological spaces follows \cite{ss}. We introduce the \rc-space concept, which we think has not been studied and described in the literature so far.  Therefore, we have limited our results to issues that require the  geometric properties of the plane. 

\section{On \rc-spaces}
If $X$ is a normal space, then the family of all open sets in $X$ fulfils both conditions from Frink's characterization, i.e., the topology constitutes a normal base. Indeed, if $x\in U \subseteq X$, then the open set $X\setminus \{ x\}=V$  is enough for $(1)$ to be fulfilled.  The condition $(2)$ is just a form of the definition of normality.
  It appears to us that there is a gap in the literature, since we could not find any information concerning a space for which the family  of all regular open sets is a normal base. Note that  a union of two regular open sets may not  be regular open, so omitting the assumption that a normal (regular)  base has to be  a ring is a significant modification, which we introduce for issues discussed here. 
We  say that a regular space is an $\rc$\emph{-space}, if  every two disjoint  regular closed subsets have disjoint open neighbourhoods. Obviously, any normal space is an $\rc$-space. We assume that an \rc-space is a regular space, so if $x\in U$, where the set $U$ is regular open, then there exists a regular open set $W$ such that $x \in W \subseteq \cl W \subseteq U$, putting $V = X \setminus \cl W$ we have verified $(1)$.  By Frink's characterization, we get   that any $\rc$-space $ X $ is completely regular.  
There exist examples of $T_1 $-spaces with bases consisting of closed-open sets, i.e., examples of completely regular spaces with  regular normal bases, 
which are not $\rc$-spaces.  These examples are completely  regular spaces with a one-point extension to a regular space, which is not completely regular, for example,   spaces considering in \cite{mys} or \cite{kp}, also counterexamples constructed by the method initiated in \cite{jon}.

\begin{pro}
 Every regular one-point extension of an $\rc$-space is completely regular. \end{pro}
\begin{proof} Let $Y= X \cup \{\infty\}$ be a regular $T_1$-space such that its subspace $X$ is an $\rc$-space.  Suppose $F\subseteq Y$ is a closed set and let $p\in Y\setminus F$ and then choose an open set $V$ such that $$p\in V \subseteq \cl V \subseteq Y\setminus F.$$ If $p\not=\infty$, then choose  an open set $W$ such that $\infty \in W$ and $p \notin \cl W $.   If $f\colon X \to [0,1]$ is a continuous function such that $f(p) =0$ and $X \cap (F \cup \cl W) \subseteq f^{-1} (1)$, then  we get a continuous extension of $f$, putting $f(\infty)=1$.   But if $p=\infty$, then   choose open set $U$ such that $$ p\in U \subseteq \cl U \subseteq V \subseteq \cl V \subseteq Y\setminus F. $$ Next,  repeat the usual proof of  Urysohn's lemma -- compare \cite[Ex. 1.5.G]{en}, starting in the first step from the disjoint regular closed sets  $\cl U $ and $\cl (X \setminus \cl V)$. 
\end{proof}

\section{The Niemytzki plane is an \rc-space}	
Recall, compare \cite[p. 39]{en}, that the Niemytzki plane $\La = \Ra \times [0,\infty)$ is a closed half-plane which is endowed with the topology generated by open discs disjoint with the real axis $\La_1=\{(x,0)\colon x \in \Ra\}$	and all sets of the form $\{a\} \cup D$ where $D\subseteq \La$ is an open disc which is tangent to $\La_1$ at the point $a \in \La_1$.

For methods relevant to normal spaces, compare \cite {en}.
Interesting discussion of the lack of normality of the Niemytzki plane is presented in~\cite{cho}.
We shall use the following notation.
If $(x,y) \in \La_2 =\La \setminus \La_1$ and $\alpha >0$, then let $\Ka ((x,y), \alpha )$ denote the intersection of $\La$ and the open disc centred at $(x,y)$ and of radius $\alpha$. 
By $\Ka ((x,0), \alpha )$ we denote the union of the one-point set $\{(x,0)\}$ and the open disc centred at $(x,\alpha)$ and of radius $\alpha$. 
Using elementary properties of the plane, one immediately checks the validity of the following fact.

\begin{fac}\label{f1}
 Let $\{(x_n,y_n)\}_{n< \omega} $, where $y_n \geqslant 0$ and $x_n \in \Ra$, be a sequence which converges to a point $(x,y)$ with respect to the Euclidean topology.
 If $ \alpha >0$, then 
$$ \hspace{14mm}\textstyle \Ka ((x,y), \frac{\alpha}{2} ) \setminus\{(x,y)\} \subseteq \bigcup\{ \Ka ((x_n,y_n), \alpha )\colon n<\omega\} . \hspace{12mm} \qed $$
\end{fac}

The below picture illustrates  a proof of Fact \ref{f1} for a case $y=0$.

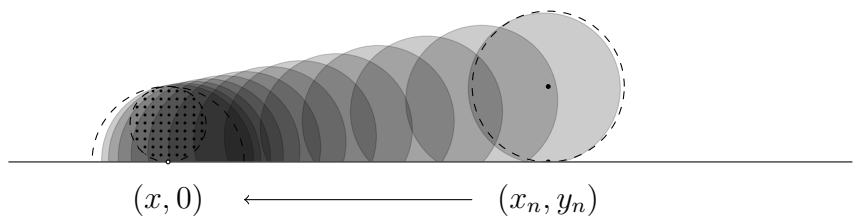
\begin{center}
 \begin{tikzpicture}
 
  \draw[dashed] (7,2) circle [radius=1];
	
	\draw[fill=black] (7,2) circle [radius=.025];
  \foreach \x in {5,...,17}
  \draw[fill=black,opacity=.2]  (2+\x*\x*\x*.001,0.4+\x*\x*\x*.0002+0.6) circle [radius=1];
  \draw[fill=black] (7,1) circle [radius=.025];

	\draw[fill=black] (2,1) circle [radius=.025];
    \draw[dashed] (2,1) circle [radius=1];
		\draw[dashed,pattern=dots] (2,1.5) circle [radius=0.5];
	 \draw[fill=white,white] (0,0) rectangle (10,1);
		\draw[-] (-.1,1) -- (11,1);
		\draw (7,.5) node  {$(x_n,y_n)$};
  \draw[->] (6,.5) -- (3,.5);
	\draw (2,.5) node  {$(x,0)$};
	\draw[fill=white] (2,1) circle [radius=.025];
 \end{tikzpicture}
\end{center}

But if $y>0$, then the set $\Ka ((x,y), \frac{\alpha}{2} ) \setminus\{(x,y)\}$ can be enlarged to $\Ka ((x,y), \alpha)$.

If subsets $F,G\subseteq\La$ are fixed, then for every $\alpha >0$ we set
\[
 F_\alpha= \{ (x,y) \in F\colon \Ka ((x,y), \alpha) \cap G= \emptyset \}.
\]

\begin{lem}\label{l2}
 Suppose $F$ and $G$ are closed and disjoint subsets of the Niemytzki plane.
 If $G$ is regular closed, then the closure of $F_\alpha$, with respect to the Euclidean topology, is disjoint from $G$.
\end{lem}

\begin{proof}
 Consider a point $(x,y)$ which belongs to the closure of $F_\alpha$ with respect to Euclidean topology.
 Since $G$ is regular closed, using Fact \ref{f1}, we obtain
 \[\textstyle
 G\cap \Ka((x,y), \frac{\alpha}{2})\setminus\{(x,y)\} =\emptyset.
 \]
 But $ \Ka((x,y), \frac{\alpha}{2})$ is an open neighbourhood of the point $(x,y) \in \La$, hence $(x,y) \notin G$.
\end{proof}
	
In the next lemma we use the following notation.
If $(x,y)\in \La$ and $\mu >0$, then let $\Ca((x,y), \mu) $ be the circle centred at $(x,y)$ of radius $\mu.$

\begin{lem}\label{l3}
  Suppose $F$ and $G$ are closed and disjoint subsets of the Niemytzki plane and let $0<\varepsilon <1$ and $0<\alpha$.
  If $G$ is regular closed, then the closure of
  \[
    \bigcup \{\Ka((x,y), \varepsilon\alpha)\colon (x,y)\in F_{\alpha}\},
  \]
  with respect to the Niemytzki plane, is disjoint from $G$.
\end{lem}

\begin{proof}
  Fix numbers $\alpha$, $\varepsilon$ and sets $F$, $F_\alpha$, $G$ as in the assumptions.
  For any point $(p,q)\in G,$ we shall find a number $\gamma >0$ such that $(x,y) \in F_\alpha$ implies
  \[
    \Ka ((x,y), \varepsilon\alpha) \cap \Ka((p,q), \gamma)= \emptyset.
  \]
  To find the appropriate $\gamma >0$, just make $\gamma$ meet the following restrictions.
  If $y>0$, then one checks that the inequality $2\gamma \leqslant \alpha -\varepsilon\alpha$ is sufficient, since $(p,q) \notin \Ka((x,y), \alpha)$.
  When $(x,0) \in F_\alpha$, choose a number $\beta >0$ such that $\beta < \dist (F_\alpha, (p,q))$ and $\beta <\varepsilon\alpha$, using Lemma \ref{l2}.
  Fix $(a,b)\in \Ca ((x,0), \beta) \cap \Ca((x,\alpha), \alpha).$
  Let $\Ea$ be the line passing through the points $(x,\varepsilon\alpha)$ and $(a,b)$.
  Fixing $(c,d)\in \Ea \cap \Ca((x,\varepsilon\alpha), \varepsilon\alpha)$, one checks that the inequality $ 2\gamma <|a-c| $ is sufficient.
\end{proof}

\begin{thm}\label{t4}
  The Niemytzki plane is an \rc-space.
\end{thm}
\begin{proof}
  Let $F$ and $G$ be regular closed and disjoint subsets of the Niemytzki plane and let $0<\varepsilon <1$.
  Consider open sets
  \[\textstyle
    W_n= \bigcup \{\Ka((x,y), \frac{\varepsilon}{n})\colon (x,y)\in F_{\frac1n}\}
  \]
  and
  \[\textstyle
    V_n= \bigcup \{\Ka((x,y), \frac{\varepsilon}{n})\colon (x,y)\in G_{\frac1n}\}.
  \]
  By Lemma \ref{l3}, we have $G\cap \cl_\La W_n = \emptyset$ and $F\cap \cl_\La V_n = \emptyset $.
  Just like in the proof of \cite[Lemma 1.5.14]{en}, we obtain open and disjoint sets
  \[
    \bigcup \{W_n\setminus (\cl_\La V_1 \cup \cl_\La V_2 \cup \ldots \cup \cl_\La V_n ) \colon n>0\} \supseteq F
  \]
  and $\bigcup \{V_n\setminus (\cl_\La W_1 \cup \cl_\La W_2 \cup \ldots \cup \cl_\La W_n ) \colon n>0\} \supseteq G$.
  \end{proof}
\section{The Sorgenfrey plane is an \rc-space} 
 Recall that the Sorgenfrey line $\Sa$ is the real line $\Ra$ with the topology generated by half-closed intervals of the form $[x,y)$: in other words, one can consider $\Sa$ as the reals $\Ra$ with the arrow topology.
The Cartesian product $ \Sa \times \Sa =\Sa^2$ equipped with the product topology is usually called the Sorgenfrey plane, compare \cite[p. 103]{ss}.
In this section we show that for any countable $\frak m$ the product space $\Sa^\frak m$ is an \rc-space, despite the fact that for $\frak m > 1$ it is not a normal space.
Our argumentation, although it is a modified discussion from the previous section, requires some adjustments and interpretations. Namely, fix a countable cardinal number $\frak m >0$.
Let $\Ra^\frak m$ be equipped with the product topology.
Thus $\Ra^\frak m$ is a metric space, since $\frak m$ is countable. We now   proceed to use the short-cut ${\bfx} = \{x_k\}_{0\leqslant k<\frak m},$ for any point $\bfx \in \Sa^\frak m$. 
If $n>0$,  then put  $\bfx + \frac1n=\{x_k+\frac1n\}_{0\leqslant k<\frak m}$ and
\[
  [{\bf x}, {\bf q})_n =\{{\bf y}\in \Sa^\frak m\colon x_i
  \leqslant y_i <q_i , \mbox{ whenever } 0\leqslant i < \min\{n, \frak m\} \},
\]
and then put
\[
  \Pa (\textbf{x}, n) = \begin{cases}
   [{\bf x}, {\bf x} +\frac1n)_n, & \mbox{if } \frak m <n ; \\
    [{\bf x}, {\bf x} +\frac1n)_n \times \prod_{n\leqslant k}\Sa_k, &\mbox{if }  n \leqslant \frak m \text{ and } \Sa_k = \Sa.
  \end{cases}
\]
The sets $ \Pa (\textbf{x}, n)$ constitute a base for $\Sa^\frak m$. 

\begin{fac}\label{sf1}
  Let $\{\bfx_k \}_{k>0}$ be a sequence  which converges to a point $\bfx $ with respect to  $\Ra^\frak m$.
  If $ n >0$, then
  \begin{center}
    \hfill$\int_{\Ra^\frak m} \Pa (\textbf{x}, n) \subseteq \bigcup\{ \Pa (\textbf{x}_k,n)\colon k > 0\}.$\hfill\qed
  \end{center}
\end{fac}
	
Given sets $F, G \subseteq \Sa^{\frak m}$ and a natural number $n>0$,  put 
$$F_{G,n}= \{ \textbf{x} \in F\colon \Pa (\textbf{x}, n) \cap G= \emptyset \}.$$
\begin{lem}\label{l6}
  Suppose $F$ and $G$ are closed and disjoint subsets of $\Sa^\frak m$.
  If $G$ is regular closed, with respect to $\Sa^\frak m$, then the closure of $F_{G,n}$, with respect to $\Ra^\frak m$, is disjoint from $G$.
\end{lem}
\begin{proof}
  Consider a point $\bfx\in G$.
  We shall find a natural number $m>0$ such that
  \[
    \Pa (\bfx, m) \cap \Pa(\bfy, 2n) =\emptyset,
  \]
  for any $\textbf{y}\in F_{G,n}$.
  Since $G$ is regular closed, we can find a base set
  \[
    \Pa({\bf p}, i)\subset G \cap \Pa(\bfx, 2n)
  \]
  such that $i\geqslant 2n $ and $x_k <p_k$ for $0\leqslant k < \min\{2n, \frak m\}$.
  Thus we get that if $\bfy \in [\bfx -\frac{1}{2n}, {\bf p}+\frac1i)_{2n}$, then $\Pa (\bfy,n)\cap \Pa({\bf p}, i)\not=\emptyset$.
  Therefore, if $\bfy \in [\bfx -\frac{1}{2n}, {\bf p}+\frac1i)_{2n}$, then $\bfy \notin F_{G,n}$.
  Choosing $m>0$ such that
  \[\textstyle
    \frac1m < p_k - x_k, \mbox{ for } 0\leqslant k <\min\{2n, \frak m\},
  \]
  we obtain
  \[
    \Pa(\bfy, 2n) \cap \Pa(\bfx, m) =\emptyset,
  \]
  for any $\bfy \in F_{G,n}$.
\end{proof}

\begin{thm}\label{t7}
  If $\frak m$ is a countable cardinal number, then the space $\Sa^\frak m$ is 
	\rc-space.
\end{thm}
\begin{proof}
  Let $F$ and $G$ be regular closed and disjoint subsets of $\Sa^\frak m$.
  Consider open sets
  \[
    W_n= \bigcup \{\Pa(\bfx, 2n) \colon \bfx  \in F_{G,n}\}
  \]
  and
  \[
    V_n= \bigcup \{\Pa(\bfx, 2n)\colon \bfx\in G_{F,n}\}.
  \]
  By Lemma \ref{l6}, we get $G\cap \cl_{\Sa^\frak m}W_n= \emptyset $ and $F\cap \cl_{\Sa^\frak m}V_n= \emptyset$.
  Just like in the proof of Theorem \ref{t4} or \cite[Lemma 1.5.14]{en}, we are done.
\end{proof}

\begin{cor}
  The Sorgenfrey plane is an \rc-space.
\end{cor}
\begin{proof}
  The corollary is special case of Theorem \ref{t7},  as illustrated in the following figure.
      \tikzfading[name=fade right,top color=transparent!0, bottom color=transparent!100]%
  \begin{center}
    \begin{tikzpicture}
      \fill[green!30!blue] (4.5,.7) rectangle (12.5,8.7);
      \draw[white,fill,path fading=fade right,fading transform={rotate=90}] (3,0) rectangle (14,8);%
      \draw[white,fill,path fading=fade right] (4.5,.8) rectangle (13,8.7);%
      \draw[pattern color=gray,pattern=north west lines] (6.5,9) -- (4,4) --(7.7, 5) --(7.7, 9);
      
      \otd{4.5}{.7}8{{\bfz} \notin F_{G,n}}
      \otc444{}
      \otc008{(x_1-\frac1{2n}, x_2-\frac1{2n})}
      \ot{7.8}{3.1}4{\bfy \in F_{G,n}}{\Pa(\bfy,2n)}
      \ot{5.9}{5.8}{1.7}{\bfp=(p_1,p_2)}{\Pa(\bfp,i)}
      \draw (8.8,7.49) node[fill=white,scale=0.7] {$(p_1+\frac1i,p_2+\frac1i)$};
      \ot44{1.5}{\bfx=(x_1,x_2)}{\Pa(\bfx, m)}
      \draw (8.55,8.03) node[scale=0.7] {$\bfx +\frac1{2n}$};
      \draw (12.98,8.7) node[scale=0.7] {${\bf z}+\frac1n$};
      \draw (7,8.4) node[fill=white,scale=0.9] {$G$};
      \draw (10.5,1.7) node[fill=white,scale=1] {$\Pa(\bfz, n)$};
    \end{tikzpicture}
  \end{center}
If ${\bf z}\in [x_1-\frac1{2n},p_1+\frac1{i})\times [x_2-\frac1{2n}, p_2+\frac1{i})$, then $G\cap \Pa({\bf z}, n)\not=\emptyset,$ hence $z\notin F_{G,n}.$ 
But if $$\bfy\notin (x_1-\frac1{2n},x_1+\frac1{m})\times (x_2-\frac1{2n}, x_2+\frac1{m}),$$  then $\Pa(\bfy, 2n) \cap \Pa(\bfx,m)=\emptyset.$
\end{proof}

\section{More examples of completely regular spaces which are not \rc-spaces  }

 Let $\La^1 \oplus \La^2 \oplus \La^3$  be the sum of three copies of the Niemytzki plane, for the definition of the sum of spaces see \cite[p. 103]{en}. Consider a quotient $X$ of this sum, obtained  by gluing copies of the rationals $\Qa\subset  \La_1 \subset \La^1$ and  $\Qa\subset  \La_1 \subset \La^2$ and copies of irrationals $\Ia\subset  \La_1 \subset \La^2$ and  $\Ia\subset  \La_1 \subset \La^3$. This quotient is a completely regular space which is not an \rc-space. Indeed, subspaces $\La^1\subset X$ and $\La^3\subset X$ are regular closed and if $V\subset X$ and $U \subset X$ are open sets such that $\La^1 \subset V$ and $\La^3 \subset U$, then the intersection $\La^2\cap U \cap V $ is non-empty, compare \cite{cho} or \cite[pp. 101--102]{ss}.

  Let us note that the above construction of a quotient space $X$ relies in simplification of the constructions begun in the paper \cite{jon}. Of course, by analogy one  can get many examples which are not \rc-spaces, using other  completely regular spaces which are not  normal. For example, let $\Sa_1 \oplus \Sa_2 \oplus \Sa_3$  be the sum of three copies of the plane with the half-open square topology, compare \cite[p. 103]{ss}. Consider a quotient $Y$ of this sum, obtained by gluing copies of the rationals $ \Qa = \{(\alpha, -\alpha): \alpha \mbox{ is a rational number}  \} \subset  \Sa_1$ and $ \Qa \subset \Sa_2$ and   copies of irrationals $\Ia= \{(\alpha, -\alpha): \alpha \mbox{ is a irrational number}\}\subset  \Sa_2$ and  $\Ia\subset   \Sa_3$. According to a similar argument as for $X $ above,  the quotient space $Y$  is  completely regular and is not an \rc-space.

\end{document}